\documentclass[12pt,centertags,oneside]{amsart}

\usepackage{amsmath,amstext,amsthm,amscd,typearea,hyperref,stmaryrd,amssymb}
\usepackage{amssymb}
\usepackage{a4wide}
\usepackage[mathscr]{eucal}
\usepackage{mathrsfs}
\usepackage{typearea}
\usepackage{charter}
\usepackage{pdfsync}
\usepackage[a4paper,width=16.2cm,top=3cm,bottom=3cm]{geometry}

\numberwithin{equation}{section}

\newtheorem{theorem}{Theorem}[section]
\newtheorem{definition}[theorem]{Definition}
\newtheorem{proposition}[theorem]{Proposition}
\newtheorem{corollary}[theorem]{Corollary}
\newtheorem{lemma}[theorem]{Lemma}
\newtheorem{remark}[theorem]{Remark}

\newcommand{\cali}[1]{\mathscr{#1}}

\newcommand{\ddc}{{dd^c}}
\newcommand{\dc}{{d^c}}
\newcommand{\dbar}{{\overline\partial}}
\newcommand{\ddbar}{{\partial\overline\partial}}

\newcommand{\PSH}{{\rm PSH}}

\newcommand{\capa}{{\rm cap}}

\newcommand{\Cc}{\cali{C}}

\newcommand{\C}{\mathbb{C}}

\newcommand{\R}{\mathbb{R}}

\renewcommand{\P}{\mathbb{P}}

\newcommand{\cH}{\mathcal{H}}

\newcommand{\1}{{\bf 1}}

\usepackage{verbatim}
\usepackage{tikz}

\renewcommand{\[}{\begin{equation}}
\renewcommand{\]}{\end{equation}}

\newcommand{\red}[1]{{\color{red}#1}}

\title[Complex Sobolev space and Monge-Amp\`ere equations]{The complex Sobolev space and H\"older continuous solutions to Monge-Amp\`ere equations}

\author{Tien-Cuong Dinh}
\address{Department of Mathematics, National University 
of Singapore, 10 Lower Kent Ridge Road, Singapore 119076.}
\email{matdtc@nus.edu.sg}

\author{S\l awomir Ko\l odziej}
\address{Faculty of Mathematics and Computer Science, Jagiellonian University, \L ojasiewicza 6, 30-348 Krak\'ow, Poland.}
\email{slawomir.kolodziej@im.uj.edu.pl}

\author{Ngoc Cuong Nguyen}
\address{Department of Mathematical Sciences, KAIST, 291 Daehak-ro, Yuseong-gu, Daejeon
34141, South Korea.}
\email{cuongnn@kaist.ac.kr} 

\begin{document}

\maketitle

\begin{abstract}
Let $X$ be a compact K\"ahler manifold of dimension $n$ and $\omega$ a K\"ahler form on $X$. We consider the complex Monge-Amp\`ere equation 
$(\ddc u+\omega)^n=\mu$, where $\mu$ is a given positive measure on $X$ of suitable mass and $u$ is an $\omega$-plurisubharmonic function. We show that the equation admits a H\"older continuous solution {\it if and only if} the measure $\mu$, seen as a functional on a complex Sobolev space $W^*(X)$, is H\"older continuous. A similar result is also obtained for the complex Monge-Amp\`ere equations on domains of $\C^n$.
\end{abstract}

\medskip\medskip


 \medskip



\section{Introduction} \label{s:intro}

Let $X$ be a compact K\"ahler manifold of dimension $n$ and $\omega$ a K\"ahler form on $X$. 
Throughout the paper, we assume that $\int_X\omega^n=1$ but we can easily extend the results to the case where $\omega$ is not normalized.
Recall that $d=\partial +\dbar$, $d^c:={1\over 2i\pi}(\partial -\dbar)$ and $\ddc={i\over \pi}\ddbar$. Recall also that 
a {\it quasi-plurisubharmonic} ({\it quasi-psh} for short) function on $X$ is locally the difference of a psh function and a smooth one. A quasi-psh function $u$ on $X$ is said to be {\it $\omega$-psh} if $\ddc u\geq -\omega$ or equivalently $\ddc u+\omega$ is a closed positive  $(1,1)$-current.
When $u$ is a bounded $\omega$-psh function, the wedge-product $(\ddc u+\omega)^n$ is a well-defined positive measure on $X$, see e.g. \cite{BT82, ko05}. 

In this article, we consider the complex Monge-Amp\`ere equation 
$$(\ddc u+\omega)^n=\mu,$$  
where $\mu$ is a given positive measure and $u$ is a bounded $\omega$-psh function on $X$ (all measures we consider in this paper are Borel measures of finite mass). 
This important equation plays a central role in complex geometry.
By taking the integral on $X$ and using Stokes' theorem, we deduce from the equation that the mass of $\mu$ satisfies
\begin{equation*} 
\|\mu\|=\int_X\omega^n=1.
\end{equation*}

From now on, we assume this condition which is necessary to solve the above equation.
Our aim is to give a new criterion so that this equation admits a H\"older continuous solution $u$.
A famous classical theorem of Yau says that when $\mu$ is given by a smooth volume form then the equation admits a smooth solution \cite{Yau}. More recently, the case of continuous and H\"older continuous solutions has been intensively studied. We refer the reader to \cite{DDGKPZ14, Dinew, DZ, DN, GKZ08, Hiep, ko98, ko08, KN, N18, N20, Vu} for more details. Some of these results will be recalled later in the present article.

Denote by $W^{1,2}(X)$ the Sobolev space of real valued functions $f$ on $X$ such that both $f$ and $df$ are of class $L^2$. Recall that {\it the complex Sobolev space} 
$W^*(X)$, introduced by Sibony and the first author in \cite{DS}, is the space of all functions $f\in W^{1,2}(X)$ such that 
\begin{equation*} 
df\wedge d^c f\leq T
\end{equation*}
for some closed positive $(1,1)$-current $T$ on $X$.
By \cite{Vig}, this is a Banach space with the norm
$$\|f\|_*:=\|f\|_{L^1(X)} + \min\big\{ \|T\|^{1/2} : \ T \text{ as above}\big\}.$$
Here, the mass of $T$ is defined by $\|T\|:=\|T\wedge\omega^{n-1}\|:=\langle T,\omega^{n-1}\rangle$.
Note that by Poincar\'e-Wirtinger inequality, if we replace $L^1(X)$ in the definition of $\|\cdot\|_*$ by $L^2(X)$, we obtain an equivalent norm.
On each bounded subset of $W^*(X)$ for $\|\cdot\|_*$ norm,
the $L^1$ norm induces a natural distance whose topology coincides with the weak topology, see \cite{DMV}. 

\begin{definition} \rm \label{d:W-Holder}
Let $\mu$ be a measure of finite mass on $X$. We say that $\mu$ is {\it $W^*(X)$-H\"older continuous} if there are positive constants $c$ and $\alpha$ such that 
$$|\mu(f)|\leq c \|f\|_{L^1(X)}^\alpha$$
for every smooth function $f$ on $X$ such that $\|f\|_*\leq 1$. Here, $\mu(f)$ denotes the usual integral of $f$ with respect to $\mu$.
\end{definition}

We will see later in Remark \ref{r:final} below that such a measure $\mu$ extends to a functional on whole $W^*(X)$ and the above inequality holds for all $f\in W^*(X)$. So $\mu$ is $W^*(X)$-H\"older continuous if and only if it defines a functional on $W^*(X)$ which is H\"older continuous with respect to the $L^1$-distance on every $\|\cdot\|_*$-bounded set. Here is our main result.

\begin{theorem}\label{t:main} 
Let $X$ be a compact K\"ahler manifold of dimension $n$ and $\omega$ a K\"ahler form on $X$ normalized so that $\int_X\omega^n=1$. 
Let $\mu$ be a probability measure on $X$. Then 
$\mu$ is $W^*(X)$-H\"older continuous if and only if there exists a H\"older continuous $\omega$-psh function $u$ solving the complex Monge-Amp\`ere equation $(dd^c u+\omega)^n=\mu$. 
\end{theorem}

We also obtain the following  local version of this theorem. Consider a smooth bounded and strictly pseudoconvex domain $\Omega$ in $\C^n$. 
Let $\omega$ be the standard K\"ahler form on $\C^n$. The Sobolev space $W^{1,2}(\Omega)$ consists of real valued functions $f$ such that $f$ and the coefficients of $df$ belong to $L^2(\Omega)$.  We define {\it the complex Sobolev space} $W^*(\Omega)$ as the space of all functions $f\in W^{1,2}(\Omega)$ such that 
$$df \wedge d^c f \leq T$$
for some closed positive $(1,1)$-current of finite mass $T$ on $\Omega$. This is also a Banach space for the norm
$$\|f\|_{*,\Omega} = \|f\|_{L^1(\Omega)} + \min \big\{\|T\|^{1/2}_{\Omega} : T \text{ as above} \big\},$$
where the mass of $T$ is defined by $\|T\|_\Omega := \|T\wedge\omega^{n-1}\|:=\int_\Omega T \wedge \omega^{n-1}$, see \cite{Vig}.

\begin{definition} \rm \label{d:W-Holder-local}
Let $\mu$ be a Borel measure of finite mass  and compact support on $\Omega$. We say that $\mu$ is {\it $W^*(\Omega)$-H\"older continuous} if there are positive constants $c$ and $\alpha$ such that 
$$|\mu(f)|\leq c \|f\|_{L^1(\Omega)}^\alpha$$
for every smooth function $f$ on $\Omega$ such that $\|f\|_{*,\Omega}\leq 1$. 
\end{definition}

Here is our second main theorem.

\begin{theorem}\label{t:main-local} 
Let $\Omega$ be a smooth bounded and strictly pseudoconvex domain in $\C^n$.
Let $\mu$ be a positive Borel measure with finite mass and compact support in $\Omega$. Then,  $\mu$ is $W^*(\Omega)$-H\"older continuous 
if and only if $\mu=(dd^c u)^n$ on $\Omega$ for a H\"older continuous function $u$ in $\overline\Omega$ which is psh on $\Omega$ and vanishes 
on the boundary $b\Omega$ of $\Omega$. 
\end{theorem}

In order to prove that $(\ddc u+\omega)^n$ (resp. $(\ddc u)^n$ in the local setting), with $u$ H\"older continuous, satisfies the $W^*$-H\"older continuity in Definition \ref{d:W-Holder} (resp. Definition \ref{d:W-Holder-local}), we use some idea from the theory of interpolation between Banach spaces. More precisely, we write $u=u_\epsilon+(u-u_\epsilon)$ for a suitable regularization $u_\epsilon$ of $u$ for which we control the norms $\|u_\epsilon\|_{\Cc^2}$ and $\|u-u_\epsilon\|_\infty$. We then use the techniques of integration by parts in order to obtain the desired estimates. 

The converse implications are more delicate. Assuming the $W^*$-H\"older continuity of $\mu$ or the H\"older continuity of $\mu$ on a suitable family of quasi-psh functions, we first show a continuity of $\mu$ with respect to some notion of capacity, see Proposition \ref{p:h-alpha} below. This allows us to apply a result in \cite{KN} to conclude. Alternatively, we can solve the complex Monge-Amp\`ere equation with a continuous solution $u$, see Theorem \ref{t:MA-C0} below. We then use again the $W^*$-H\"older continuity of $\mu$ to obtain an estimate on $\|\rho_\delta u -u\|_\infty$, where 
$\rho_\delta u$ is a regularization of $u$ introduced by Demailly \cite{Demailly}. This estimate implies that $u$ is in fact H\"older continuous.

The paper is organized as follows. In Section \ref{s:known}, we will recall some known results and describe the structure of the proof. In Section \ref{s:MA-C0}, we discuss some sufficient conditions on $\mu$ so that the complex Monge-Amp\`ere equation admits a continuous solution $u$. 
Section \ref{s:proof} is devoted to the proofs of the main results stated above. Finally, in Section 5, a natural capacity for the  complex Sobolev space $W^*$,   studied by Vigny \cite{Vig},  is estimated in terms of the Bedford-Taylor capacity. This supplements Vigny's result, which gives a bound in the other direction.

\medskip\noindent 
\textbf{Acknowledgements.} 
The first author is supported by the NUS and MOE grants R-146-000-248-114 and MOE-T2EP20120-0010.
The second and third authors are partially supported by NCN grant  
2017/27/B/ ST1/01145. The third author is also partially supported by the start-up grant G04190056 of KAIST and the National Research Foundation of Korea (NRF) grant No. 2021R1F1A1048185.

\section{Structure of the proof and some partial results} \label{s:known}

Let $X,n$ and $\omega$ be as in the Introduction. Denote by $\PSH(X,\omega)$ the cone of all $\omega$-psh functions on $X$.
Denote also by $\PSH(X,\omega,[-1,0])$ the set of $\omega$-psh functions $v$ such that $-1\leq v\leq 0$. 

\begin{definition} \rm \label{d:Holder-PSH}
A positive measure $\mu$ on $X$ is said to be 
{\it $\PSH(X,\omega)$-H\"older continuous} (resp. {\it $\PSH(X,\omega,[-1,0])$-H\"older continuous}) if there are positive constants $c$ and $\alpha$ such that 
$$|\mu(f_1)-\mu(f_2)|\leq c \|f_1-f_2\|_{L^1(X)}^\alpha$$
for all functions $f_i\in \PSH(X,\omega)$ with $\max_X f_i=0$ (resp. $f_i\in \PSH(X,\omega, [-1,0])$). 
\end{definition}

Note that the first notion is equivalent to say that $\mu$ has a H\"older continuous super-potential and it is stronger than the second notion, 
see \cite{DN} for details.

We call a {\it smooth strictly pseudoconvex coordinate patch} of $X$ any smooth domain $\Omega\subset X$ such that there is a smooth bijective map $\pi:\overline\Omega\to\overline U$ which is holomorphic on $\Omega$, where $U$ is a smooth bounded strictly pseudoconvex domain 
in $\C^n$. Denote by $\PSH(\Omega)$ the cone of all psh functions on $\Omega$. If $K$ is a subset of $X$ denote by $\1_K$ the characteristic function of $K$. The following result was obtained in \cite{DN, KN}, see also \cite{N18,N20}.

\begin{theorem} \label{t:DN-KN}  
Let $X$ be a compact K\"ahler manifold of dimension $n$ and $\omega$ a K\"ahler form on $X$ normalized so that $\int_X\omega^n=1$. 
Let $\mu$ be a probability measure on $X$. Then the following properties are equivalent.
\begin{itemize}
\item[(h1)] $\mu = (dd^c u+\omega)^n$ for a H\"older continuous $\omega$-psh function $u$ on $X$;
\item[(h2)] There is a number $0<\alpha<1$ such that for every strictly pseudoconvex coordinate patch $\Omega\subset X$ and any compact set $K \Subset \Omega$, we have $\1_K \!\cdot\!\mu = (dd^c u)^n$ for some function $u \in \PSH(\Omega)\cap C^{\alpha}(\overline\Omega)$;
\item[(h3)] $\mu$ is $\PSH(X,\omega)$-H\"older continuous.
\end{itemize}
\end{theorem}

We are going to prove the following result.

\begin{theorem} \label{t:W-Holder-MA}
Under the hypotheses of Theorem \ref{t:DN-KN}, the above properties (h1)-(h3) are all equivalent to the following ones.
\begin{itemize}
\item[(h4)] $\mu$ is $\PSH(X,\omega,[-1,0])$-H\"older continuous;
\item[(h5)] $\mu$ is $W^*(X)$-H\"older continuous;
\item[(h6)] For every strictly pseudoconvex coordinate patch $\Omega\subset X$ and any compact set $K \Subset \Omega$, the measure 
$\1_K \!\cdot\!\mu$ is $W^*(\Omega)$-H\"older continuous.
\end{itemize}
\end{theorem}

It is clear that (h6) $\Longrightarrow$ (h5). We will show in this section that (h2) $\Longrightarrow$ (h6) and in Section \ref{s:proof} we will complete the proof by showing that (h4) $\Longleftrightarrow$ (h1) and (h5) $\Longrightarrow$ (h4). It is clear that Theorem \ref{t:main} is a direct consequence of Theorems \ref{t:DN-KN} and \ref{t:W-Holder-MA}.
For Theorem \ref{t:main-local}, we take $X:=\P^n$ and identify $\Omega$ to some open set of an affine chart $\C^n\subset\P^n$. We see that Theorem \ref{t:main-local} is also a consequence of Theorems \ref{t:DN-KN} and \ref{t:W-Holder-MA}. 

As mentioned above, we prove now the implication (h2) $\Longrightarrow$ (h6). It is the direct consequence of the following proposition.

\begin{proposition} \label{p:h2-h6} 
Let $\Omega$ be a smooth bounded strictly pseudoconvex domain in $\C^n$ and $K\Subset \Omega$ a compact set. Let $u_1,\ldots,u_n$ be H\"older continuous psh functions on $\Omega$. Then 
the positive measure $\1_K \!\cdot\! dd^c u_1\wedge\ldots\wedge dd^c u_n$ is $W^*(\Omega)$-H\"older continuous.
\end{proposition}

We first state the following useful lemma that can be easily extended to the global setting. We will omit the proof as it follows from a classical smoothing argument (by setting $f^\pm:=\chi^\pm(f)$ where $\chi^\pm$ are suitable approximations of $t\mapsto \max(\pm t,0)$).

\begin{lemma} \label{l:W-positive}
Let $\Omega$ be a bounded domain in $\C^n$. Let $f$ be a smooth function on $\Omega$. Then for every $\varepsilon>0$ there are smooth non-negative functions $f^+$ and $f^-$ such that
$$f=f^+-f^-, \qquad \|f^\pm\|_{L^1(\Omega)}\leq \|f\|_{L^1(\Omega)}+\epsilon \qquad \text{and} \qquad \|f^\pm\|_{*,\Omega}\leq \|f\|_{*,\Omega}+\epsilon.$$
In particular, if $\mu$ and $\mu'$ are positive measures with compact supports in $\Omega$ such that $\mu\leq \mu'$ and $\mu'$ is $W^*(\Omega)$-H\"older continuous, then $\mu$ is also $W^*(\Omega)$-H\"older continuous.
\end{lemma}

\proof[Proof of Proposition \ref{p:h2-h6}] 
By hypothesis, there is a smooth strictly psh function $\rho$ on a neighbourhood of $\overline\Omega$ such that $\Omega=\{\rho<0\}$. 
By reducing slightly the domain $\Omega$, we can assume that $u_k$ is defined on a neighbourhood of $\overline \Omega$. Replacing 
$u_k$ by $\max\{u_k-A_1, A_2\rho\}$ and $\rho$ by $A_2\rho$ for some positive constants $A_2\gg A_1 > \sup_\Omega u_k$, we can assume that $u_k= \rho$ on a neighbourhood of $b\Omega$. This doesn't change the values of $u$ on a neighbourhood of $K$.

Let $f$ be a non-negative smooth function on $\Omega$ such that $\|f\|_{*,\Omega}\leq 1$. By Lemma \ref{l:W-positive}, we only need to show that 
\begin{equation*} 
\int_K f \ddc u_1\wedge\ldots\wedge \ddc u_n \leq c \|f\|_{L^1(\Omega)}^\alpha
\end{equation*}
for some positive constants $c$ and $\alpha$ independent of $f$.
For this purpose, we will prove by induction on $k=0,...,n$ that given any compact set $K\subset\Omega$ we have 
\begin{equation} \label{e:K-int-MA-bis}
\int_K f \ddc u_1\wedge\ldots\wedge \ddc u_k\wedge \omega^{n-k} \leq c \|f\|_{L^1(\Omega)}^\alpha
\end{equation}
for some positive constants $c$ and $\alpha$ independent of $f$. Recall that $\omega$ is the standard K\"ahler form on $\C^n$.

Obviously, the estimate holds for $k=0$. Assume that the desired estimate is true for $k-1$ instead of $k$. It remains to show the same property for $k$, i.e. to prove \eqref{e:K-int-MA-bis}.
For simplicity and in order to avoid confusion, let us write 
$$S:= dd^c u_1\wedge\ldots\wedge \ddc u_{k-1} \wedge \omega^{n-k} \quad \text{and} \quad u:=u_k$$ 
and read the induction hypothesis as: 
given any compact set $K'\subset\Omega$ we have 
\begin{equation} \label{e:K-int-MA-3}
\int_{K'} f S\wedge \omega \leq c_0 \delta^{4\kappa} \qquad \text{with} \qquad \delta:=\|f\|_{L^1(\Omega)}\leq 1
\end{equation}
for some positive constants $c_0$ and $\kappa$ independent of $f$. 

Fix a compact set $L$ and a domain $\Omega^+$ such that $K\subset L\Subset \Omega\Subset \Omega^+$ and $u=\rho$ on  $\Omega^+\setminus L$. 
Fix also a cutoff smooth function $0\leq\chi\leq 1$ with compact support in $\Omega$ such that $\chi=1$ on a domain $\Omega^-$ with $L\Subset \Omega^-\Subset \Omega$. Denote by $\lambda$ the distance between $b\Omega$ and $b(\Omega^+\setminus\Omega^-)$ and define $\varepsilon:=\lambda\delta^\kappa\leq \lambda$. 
Consider $u_\varepsilon$ the standard $\varepsilon$-regularization of $u$ by the convolution in $\Omega^+$.
Then, we have for some positive constant $c_1$
\begin{equation} \label{e:u-epsilon}
\|u\|_{C^2(\Omega \setminus \Omega^-)} \leq c_1, \qquad  \|u-u_\varepsilon\|_{C^2(\Omega \setminus \Omega^-)} \leq c_1\varepsilon \qquad \text{and} \qquad \|u_\varepsilon\|_{C^2(\Omega)} \leq  c_1\varepsilon^{-2}
\end{equation}
which imply for some positive constant $c_2$
$$-c_2\varepsilon^{-2} \omega \leq dd^c u_\varepsilon \leq c_2\varepsilon^{-2} \omega.$$
Thus, using \eqref{e:K-int-MA-3} for a compact set $K'$ containing the support of $\chi$, we get
$$\int_\Omega \chi f S \wedge dd^c u_\varepsilon  \lesssim \varepsilon^{-2} \int_\Omega \chi f S \wedge \omega 
\leq 	\varepsilon^{-2} \delta^{4\kappa}  \lesssim \delta^{2\kappa}. $$
Now, since
$$\int_K f S \wedge dd^c u \leq \int_\Omega \chi f S \wedge dd^c u = \int_\Omega \chi f S \wedge dd^c u_\varepsilon  + 
\int_\Omega \chi f S \wedge dd^c (u-u_\varepsilon),$$
it is enough to bound the last integral by a constant times a power of $\delta$. 

By  Stokes' theorem, the last integral can be written as
$$\int_\Omega \chi f S \wedge dd^c (u- u_\varepsilon)  = 	-\int_\Omega f d\chi\wedge S \wedge d^c (u-u_\varepsilon) 
-\int_\Omega \chi df\wedge S   \wedge d^c (u-u_\varepsilon).$$
Since $d\chi$ vanishes outside $\Omega\setminus\Omega^-$, the second inequality in \eqref{e:u-epsilon} implies that the first term in the RHS of the last identity is bounded by a constant times $\varepsilon=\delta^\kappa=\|f\|_{L^1(\Omega)}^\kappa$. It remains to bound the last integral of the last identity.
By the Cauchy-Schwarz inequality, this integral satisfies
\begin{equation} \label{e:df-dcu}
\Big|\int_\Omega \chi df \wedge S \wedge d^c (u-u_\varepsilon) \Big|^2 \leq \Big(\int_\Omega S \wedge df \wedge d^c f\Big)\Big( \int_\Omega \chi^2 S \wedge d(u-u_\varepsilon) \wedge d^c (u- u_\varepsilon)\Big).
\end{equation}

We show that the first factor in the last product is bounded by a constant. 
Since $\|f\|_{*,\Omega} \leq 1$, there is a closed positive $(1,1)$-current $T$ on $\Omega$ that $df \wedge d^c f \leq T$ and  $\|T\| \leq 1$. Therefore, using the definition of $S$ and the fact that $u=\rho$ near $b\Omega$, we obtain by integration by parts
$$\int_\Omega S \wedge df \wedge d^c f \leq  \int_\Omega S \wedge T = \int_\Omega (dd^c\rho)^{n-1} \wedge T$$
which is bounded by a constant because $\ddc\rho$ is bounded by a constant times $\omega$.
In order to get the result, it is enough to show that the last integral in \eqref{e:df-dcu} is bounded by a constant times a power of $\delta$.

For this purpose, by  Stokes' theorem, we obtain
\begin{eqnarray*}
\lefteqn{\int_\Omega \chi^2 S \wedge d(u-u_\varepsilon) \wedge d^c (u-u_\varepsilon) =}\\
& =  &-\int_{\Omega} 2\chi  (u-u_\varepsilon) d\chi\wedge S \wedge d^c (u-u_\varepsilon) -  	
\int_\Omega \chi^2 (u-u_\varepsilon) S \wedge  dd^c (u-u_\varepsilon).
\end{eqnarray*}
Since $d\chi=0$ outside $\Omega\setminus\Omega^-$, the second identity in \eqref{e:u-epsilon} implies that the first term in the RHS of the last line is bounded by a constant times $\varepsilon^2=\delta^{2\kappa}=\|f\|_{L^1(\Omega)}^{2\kappa}$. Moreover, the second term satisfies
\begin{eqnarray*}
\Big|\int_\Omega \chi^2 (u-u_\varepsilon) S \wedge  dd^c (u-u_\varepsilon)\Big| &\leq & 
\|u-u_\varepsilon\|_{L^\infty(\Omega ) } \int_\Omega S \wedge  dd^c (u+u_\varepsilon) \\
&=&  \|u-u_\varepsilon\|_{L^\infty(\Omega ) } \int_\Omega (\ddc\rho)^{n-1} \wedge  dd^c (\rho+\rho_\varepsilon),
\end{eqnarray*}
where the last identity is obtained by using Stokes' theorem and $\rho_\varepsilon$ is the standard $\varepsilon$-regularization of $\rho$ by the convolution.

The last integral is clearly bounded by a constant because $\rho$ is smooth. 
Finally, since $u$ is H\"older continuous,  we have 
$$\|u-u_\varepsilon\|_{L^\infty} \leq c_3 \varepsilon^\gamma = c_3 \delta^{\gamma \kappa}$$ 
for some positive constants $c_3$ and $\gamma$. The result follows. 
\endproof

Using the same techniques, we obtain the following result.

\begin{lemma} \label{l:W-MA-C0}
Let $\Omega$ be a bounded domain in $\C^n$ and $u_1,\ldots, u_n$ bounded psh functions on $\Omega$. Let $K$ be a compact subset of $\Omega$
and $\mu$ a positive measure supported by $K$. Assume that $\mu\leq \ddc u_1\wedge \ldots \wedge \ddc u_n$. Then, $\mu$ seen as a functional on the space $W^*(\Omega)\cap\Cc^0(\Omega)$, extends to a functional $\mu:W^*(\Omega)\to\R$ which is continuous in the following sense: if a bounded sequence $(f_k)\subset W^*(\Omega)$ converges in the sense of currents to a function $f\in W^*(\Omega)$ then $\mu(f_k)$ converges to $\mu(f)$. 
\end{lemma}
\proof
Let $\Omega'\Subset\Omega$ be a domain containing $K$. By \cite[Lemma 2.3]{DMV}, for every function $f\in W^*(\Omega)$, the standard regularization $f_\varepsilon$ of $f$ by the convolution, converges to $f$ when $\varepsilon$ goes to 0; moreover, $f_\varepsilon$ is smooth and has a bounded norm in $W^*(\Omega')$. Therefore, in order to obtain the lemma, it is enough to show that if a bounded sequence $(f_k)\subset W^*(\Omega')\cap \Cc^0(\Omega')$ converges in the sense of currents to a function $f\in W^*(\Omega')$, then the sequence of real numbers $\mu(f_k)$ is a Cauchy sequence.
This property is in fact a particular case of \cite[Lemma 2.6]{DMV}.
\endproof

\begin{remark} \rm \label{r:MA-int}
Any function $f\in W^*(\Omega)$ admits canonical representatives which are functions defined outside a pluripolar set and equal to $f$ almost everywhere. Two canonical representatives of $f$ are equal outside a pluripolar set. In the last lemma, the value of $\mu(f)$ is in fact equal to the usual integral of a canonical representative of $f$ with respect to $\mu$. It is well known that $\mu$ has no mass on pluripolar sets. Therefore, the last integral does not depend on the choice of the canonical representative of $f$. We refer to \cite{DMV} for more details.
\end{remark}


\section{Monge-Amp\`ere measures with continuous potentials} \label{s:MA-C0}

From now on, we work in the global setting. Some results can be easily extended to the local setting but for simplicity 
we choose  not to  discuss this case here. Let $X,n$ and $\omega$ be as in the Introduction. 

We first consider some properties of $\PSH(X,\omega)$ and $W^*(X)$.
Define 
$$\PSH_0(X,\omega):=\big\{u\in\PSH(X,\omega) : \max_X u =0\big\}.$$
The following result is due to Tian \cite{Ti}.

\begin{lemma} \label{l:Tian}
There are positive constants $c$ and $\alpha$ such that 
\begin{equation} \label{e:Tian}
\int_X e^{-\alpha u} \omega^n\leq c \quad \text{for every} \quad u\in\PSH_0(X,\omega).
\end{equation}
In particular, Tian's invariant defined as
$$\alpha_0:=\sup\big\{\alpha : \eqref{e:Tian} \text{ holds for some positive constant } c\big\}$$
is a finite positive number.
\end{lemma}

Let $E$ be a Borel subset of $X$. 
Recall from \cite{BT82, ko03} that the global version of Bedford-Taylor capacity of $E$ is defined by
$$\capa_\omega(E) := \sup\Big\{\int_E (dd^c v+\omega)^n : v\in \PSH(X,\omega), \; -1\leq v \leq 0\Big\}.$$
Recall that the {\it global extremal functions} associated with $E$ are defined as
\[\label{eq:global-local-ext}h_{E}(x) := \sup \big\{v(x): v \in \PSH(X,\omega), \; v \leq 0 \text{ and } v_{|_E} \leq -1 \big\}\]
and
\[ \label{eq:siciak-fct}V_E(x):= \sup \big\{ v(x): v\in \PSH(X,\omega), \; v\leq 0 \text{ on }E \big\}.\]
The second function is an analogue of  the Siciak-Zahariuta extremal function in $\C^n$. Denote by $h^*_E$ and $V_E^*$ the upper semi-continuous regularizations of $h_E$ and $V_E$ respectively.
The following result was obtained in \cite{GZ05}. We only state it for compact sets for simplicity.

\begin{lemma} \label{l:extremal}
Let $K$ be a compact subset of $X$. Then we have the following properties
\begin{itemize}
\item[(a)] $h_K =-1$ on $K$, $h_K^* \in \PSH(X,\omega,[-1,0])$, $h_K\leq h_K^*$ and  the set $\{h_K \not= h_K^*\}$ is pluripolar;
\item[(b)] $K$ is non-pluripolar if and only if $\sup_X V_K^*<+\infty$; and in this case, we have
$$\capa_\omega(K) = \int_K (dd^c h_K^*+\omega)^n= \int_X (-h_K^*) (dd^c h^*_K+\omega)^n$$
and also $V_K^*\in\PSH(X,\omega)$, $V_K^*\geq 0$ with 
$$\int_K  (dd^c V_K^*+\omega)^n  = \int_X (dd^c V_K^*+\omega)^n=\int_X\omega^n=1;$$
\item[(c)] there is a positive constant $c$ independent of $K$ such that if $\sup_X V_K^* \geq 1$ then 
$$\frac{1}{\capa_\omega(K)^{1/n}} \leq \sup_X V_K^* \leq \frac{c}{\capa_\omega(K)}\cdot$$
\end{itemize}
\end{lemma}
\proof
For (a) and (b), see \cite{GZ05}. Property (c) is slightly different from \cite[Prop 7.1]{GZ05}. We give here the proof for reader's convenience.
Define 
$$M:=\sup_X V_K^* \quad \text{and} \quad u:=M^{-1}(V_K^*-M).$$ 
Since $M \geq 1$ by hypothesis, we have $u\in \PSH(X,\omega,[-1,0])$. Note that $ dd^c V_K^*+\omega \leq M (dd^c u+\omega)$. Therefore, by (b)
$$M^{-n} = M^{-n} \int_K (dd^c V_K^*+\omega)^n \leq \int_K  (dd^cu+\omega)^n \leq \capa_\omega(K).$$
The first inequality is proved.

We prove now the second inequality. Let $v \in \PSH(X,\omega) $ be such that $v\leq 0$ on $K$ and define $v':=M^{-1}(v-M)$. 
Then, by the definition of $M$, we have 
$v \leq M$ on $X$. Hence $v' \in \PSH(X,\omega) $ and furthermore,  $\sup_X v' \leq 0$ and $v' \leq -1$ on $K$. This implies that $v' \leq h_K^*$. 
It follows from the definitions of $V_K$ and $V_K^*$ that
$$\frac{V_K^*-M}{M} \leq h_K^* \leq 0.$$
Then,
$$\capa_\omega(K) = \int_X (-h_K^*) (dd^c h_K^*+\omega)^n \leq \frac{1}{M} \int_X|V_K^*-M| (dd^c h_K^*+\omega)^n.$$
By \cite[Cor 3.3]{GZ05} the last integral is less than
\[ \label{eq:l1-uniform-x} \int_X |V_K^*- M| \omega^n + n \leq \sup\big\{\|v\|_{L^1} : v\in \PSH_0(X,\omega)\big\}+n =:c
\]
because $\sup_X (V_K^*-M) =0$. Since $\PSH_0(X,\omega)$ is compact in $L^1(X)$, the constant $c$ is finite.  
The second inequality in (c) follows easily.
\endproof

The following lemma shows that the cone $\PSH (X,\omega) \cap L^\infty(X)$ is contained in the space $W^*(X)$.

\begin{lemma} \label{l:w-w-norm-x}
Let $u, v\in \PSH(X,\omega) \cap L^\infty(X)$ be such that $v \leq u\leq 0$. Then we have 
$$\|u-v\|_* \leq \|u-v\|_{L^1} +  \|v\|_{L^\infty}\leq \|v\|_{L^1}+\|v\|_{L^\infty}.$$
In particular, if $u\equiv 0$, then 
$$\|v\|_* \leq \|v\|_{L^1} + \|v\|_{L^\infty}.$$
\end{lemma}

\proof 
We only need to prove the first inequality. 
Set $f= u-v \geq 0$. Then, we have $\|f\|_{L^\infty} \leq \|v\|_{L^\infty}$ as $v \leq u\leq 0$. Since
$$dd^c f^2 + 2f dd^c v = 2 df\wedge d^c f + 2f dd^c u \quad \text{and} \quad \ddc u+\omega\geq 0,$$
we have
$$	2 f (dd^c v+\omega) + dd^c f^2 \geq  2 df \wedge d^c f.$$
Using $\|f\|_{L^\infty} \leq \|v\|_{L^\infty}$ and $dd^c v+\omega\geq 0$, we deduce that
$$\|f\|_* \leq \|f\|_{L^1} +  \int_X \Big( \|v\|_{L^\infty} (\ddc v+\omega)  + \frac{1}{2}dd^cf^2 \Big)   \wedge \omega^{n-1} = \|f\|_{L^1} + \|v\|_{L^\infty},$$
where we used Stokes' theorem and the normalization of $\omega$ for the last identity.
\endproof

The following lemma is also useful for us. It can be deduced from \cite[Proposition~4.7]{DS}. We leave the proof to the readers.

\begin{lemma} \label{l:W-bounded}
Let $\mu$ be a positive measure on $X$. Assume that $\mu$ is $W^*(X)$-bounded, i.e., there exists a constant $c>0$ such that $|\mu(f)|\leq c \|f\|_*$ for every smooth function $f$ on $X$. Then $\mu$ has no mass on pluripolar subsets of $X$.
\end{lemma}

\begin{definition} \rm \label{d:W-log}
Let $\mu$ be a positive measure on $X$ and $p$ a positive number. We say that $\mu$ is {\it $(W^*(X),\log^p)$-continuous} if there is a positive constant $c$ such that $|\mu(f)|\leq c \big(\log^\star\|f\|_{L^1}\big)^{-p}$ for every smooth function $f$ on $X$ with $\|f\|_*\leq 1$, where $\log^\star:=1+|\log|$. 
\end{definition}

The following theorem gives us sufficient conditions for a measure to have a continuous Monge-Amp\`ere potential.
The implication (c2) $\Longrightarrow$ (c1) has been obtained in \cite[Theorem~2.1]{EGZ09} (see also \cite{ko98}).

\begin{theorem} \label{t:MA-C0}
Let $X$ be a compact K\"ahler manifold of dimension $n$ and $\omega$ a K\"ahler form on $X$ normalized so that $\int_X\omega^n=1$. Let $\mu$ be a probability measure on $X$. Then the implications (c3) $\Longrightarrow$ (c2) $\Longrightarrow$ (c1) hold, where
\begin{itemize}
\item[(c1)] $\mu = (dd^c u+\omega)^n$ for a continuous $\omega$-psh function $u$ on $X$;
\item[(c2)] $\mu$ belongs to the class $\cH(\alpha)$ for some $\alpha>0$, that is, we have $\mu(K)\leq c \; \capa_\omega(K)^{1+\alpha}$ 
for some constant $c>0$ and for every compact subset $K$ of $X$;
\item[(c3)] $\mu$ is $(W^*(X),\log^p)$-continuous for $p>n+1$.
\end{itemize}
\end{theorem}

We need the following lemma. Note that functions in $\PSH(X,\omega)$ are defined everywhere and upper semi-continuous. In particular, their integrals with respect to $\mu$ are meaningful.

\begin{lemma} \label{l:cap-vol}
Let $\mu$ satisfy the above property (c3). Then there  is a constant $c>0$  such that for $s\geq 1$ and $v\in \PSH_0(X,\omega)$
$$\mu \{v\leq-s\} \leq c \max\{1, \|v\|_{L^\infty} \}  \left(  \alpha_0 s  + |\log \|v\| _{L^\infty } |+ 1  \right)^{-p} ,$$
where $\alpha_0$ is the Tian's invariant.
\end{lemma}

\proof
Since $s\geq 1$, we only need to consider the case where 
$\|v\|_{L^\infty} \geq 1$, otherwise the estimate is clear.  
Put $v_s := \max\{v,-s+1\}$. Then $v\leq v_s \leq 0$ and $v_s \in \PSH_0(X,\omega)$. It follows from Lemma~\ref{l:w-w-norm-x} that
$$\|v_s - v\|_* \leq		\|v_s - v\|_{L^1} +  \|v\|_{L^\infty} \leq 2 \|v\|_{L^\infty}.$$
By Lemma~\ref{l:Tian},  we have  for every $t\geq 0$ and for $\alpha:=\alpha_0/2$
$$\int_{\{ v<-t\} } \omega ^n \leq \int_X e^{-\alpha(v+t)} \omega^n  \leq c e^{-\alpha t}.$$  
Therefore, applying this for $t = s+j$ with $j =0,1, \ldots$ we get
\[\label{eq:l1-l1-x} \begin{aligned}
\|v_s -v\|_{L^1} &= \int_{\{v <-s+1\}} |v_s -v| \omega^n \leq \sum_{j=0}^\infty  \int_{\{  -s-j \leq v<-s-j+1\}}  (j+1)   \omega^n \\ & =\  \sum_{j=0}^\infty  \int_{\{  v<-s-j+1\}}   \omega^n 
\leq c_1 \sum _{j=0 }^{\infty } e^{-\alpha j}  e^{-\alpha s} =c_2 e^{-\alpha s}
\end{aligned}\]
for some positive constants $c_1$ and $c_2$. 
Recall that $\|v_s - v\|_* \leq 2 \|v\|_{L^\infty}$. Then, the condition (c3) and Definition \ref{d:W-log} applied for $f:=(2\|v\|_{L^\infty})^{-1} (v_s-v)$ 
  imply
\[\begin{aligned}\notag
	\int_X (v_s -v) d\mu 
&\leq	 c_3  \|v\|_{L^\infty}   \Big(1+ \Big|\log \frac{\|v_s - v\|_{L^1}}{2 \|v\| _{L^\infty}} \Big|\Big)^{-p}
\end{aligned}
\]
for some positive constant $c_3$. Inserting  \eqref{eq:l1-l1-x} into the last inequality, we get
$$
\mu\{v\leq -s\} \leq \int_X (v_s -v) d\mu \leq c_4 \|v\|_{L^\infty} \left(\alpha s   + \log \|v\| _{L^\infty }  + 1  \right)^{-p}
$$
for some positive constant $c_4$. The lemma follows by choosing a suitable constant $c$.
\endproof

\proof[End of the proof of Theorem \ref{t:MA-C0}]
Assume the property (c3). We only need to show that (c2) is true.
We use an argument which is inspired by the proofs in  \cite{EGZ09} and \cite{GZ05}. Let $K\subset X$ be compact as in (c2). Since $\mu$ vanishes on pluripolar sets (Lemma \ref{l:W-bounded}) we may assume that $K$ is non-pluripolar. Let $h_K^*$ be the relative extremal function of $K$.
Define the number $\tau>0$ by 
\[\notag
	\tau^n:= \capa_\omega(K) = \int_{X}|h_K^*| (dd^ch_K^*+\omega)^n >0  \quad \text{(see Lemma \ref{l:extremal})}.
\]
Let $V_K$ be the global extremal function defined in \eqref{eq:siciak-fct}. Since  $K$ is non-pluripolar,  we have that $M:= \sup_X V_K$ is a finite number. 

\medskip\noindent
{\bf Case 1:} assume that $0\leq M \leq 1$. Then  $0\leq V_K^* \leq 1$ hence $V_K -1$ coincides with the relative extremal function $h_K$. It follows from Lemma~\ref{l:extremal}-(b) that
\[\notag
	1= \int_K (dd^c V_K^*+\omega)^n \leq \capa_\omega(K) \leq \capa_\omega(X) =1.
\]
Thus, the assertion (c2)  clearly holds in this case.

\medskip\noindent
{\bf Case 2:} assume now that $M >1$. Then, by Lemma~\ref{l:extremal}-(c)
\[\label{eq:cap-sup-x}
	\frac{1}{\tau} \leq   M \leq \frac{c}{\tau^n} \cdot
\]
We also have from \eqref{eq:global-local-ext} that the function $h:= (V_K^* -M)/M$ satisfies $h\leq h_K^*$.
Hence, using Lemma~\ref{l:cap-vol}, \eqref{eq:cap-sup-x}, and the fact that $p>n+1>1$, we have 
\begin{eqnarray} \label{eq:h-sublevel-set}
\mu \{h_K^* \leq -1\} & \leq & 	\mu \{h \leq -1 \}  \ = \  \mu\{V_K^* - M \leq -M\}  \nonumber\\
& \leq & 	c' M (\alpha_0 M +\log M +1  )^{-p}  \ \leq	\  c''  (\capa_\omega(K) )^{\frac{p-1}{n} },
\end{eqnarray}
where $c'$ and $c''$ are positive constants.

To complete the proof recall from Lemma \ref{l:extremal} that $h_K = h_K^*$ outside a pluripolar set. By Lemma~\ref{l:W-bounded}, $\mu$ vanishes on this set. Thus,
\[\label{eq:cap-ineq2} \notag
	\mu(K) \leq \mu\{h_K =-1\} = \mu\{h_K^* =-1\} \leq \mu\{h_K^* \leq -1\}.
\]
This combined with \eqref{eq:h-sublevel-set} and $p>n+1$ complete the proof of (c2).
\endproof


\section{End of the proof of the main results and some remarks} \label{s:proof}

From the discussion after Theorem \ref{t:W-Holder-MA}, we only need to show that (h4) $\Longleftrightarrow$ (h1) and  (h5) $\Longrightarrow$ (h4). This also completes the proof of Theorems \ref{t:main} and \ref{t:main-local} in the Introduction.

We will use the notion of moderate measures introduced in \cite{DNS}. Recall from \cite{DNS} that a measure $\mu$ is moderate if and only if  there exist positive constants $\alpha$ and $c$ such that for every compact set $K \subset X$,
$$\mu (K) \leq c \exp\Big(-\alpha\, \capa_\omega(K)^{-1/n}\Big).$$

\begin{proposition}\label{p:h-alpha} 
Let $\mu$ be a probability measure on $X$. Assume that $\mu$ is $\PSH(X, \omega, [-1,0])$-H\"older continuous. Then  $\mu$ is moderate. In particular, we have $\mu \in \cH(\alpha)$ for every $\alpha>0$.
\end{proposition}

\begin{proof} Let $0<\beta\leq 1$ be a number such that $\mu$ is $\beta$-H\"older on $\PSH(X, \omega, [-1,0])$. 
We will use arguments similar to the ones in the proof of  Theorem~\ref{t:MA-C0}. Instead of using Lemma~\ref{l:cap-vol}, we need the following estimate where $\alpha:=\alpha_0/2$ as above.

\smallskip\noindent
{\bf Claim.} There is a positive constant $c$ such that 
$$\mu \{v\leq -s\} \leq  c \max\{1, \|v\|_{L^\infty}^{1-\beta}\} e^{-\beta \alpha s} $$
for every $s\geq 1$ and $v\in \PSH_0(X,\omega)$. 

\smallskip

Indeed, since $s\geq 1$, we can assume that $\|v\|_{L^\infty}\geq 1$.  Applying Definition~\ref{d:Holder-PSH} to the functions $(v_s-v)/\|v\|_\infty$ and 0 with $\alpha$ replaced by $\beta$ implies that
$$
	\int_X (v_s -v) d\mu  \leq c' \|v\|_{L^\infty}^{1-\beta} \cdot \|v_s -v\|_{L^1}^\beta
$$ 
for some positive constant $c'$. The claim then follows from \eqref{eq:l1-l1-x} and the inequality $\mu(\{v\leq -s\}) \leq \int_X (v_s- v) d\mu$.

Now, observe that there is a positive constant $c$ such that for every $s>0$
$$s^{1-\beta} \exp(-\beta \alpha s) \leq c \exp (-\beta \alpha s/2).$$
Using the notations of  \eqref{eq:cap-sup-x} and applying the above inequality to $s=M$ imply that
$$\mu(K) \leq \mu\{h_K^*\leq -1\} \leq c'' \exp\Big( - \frac{\beta \alpha}{2 [\capa_\omega(K)]^{1/n}}\Big),
$$
where $c''$ is some positive constant (for every compact set $K\subset X$). Thus, $\mu$ is moderate and the proof is finished.
\end{proof}

\proof[End of the proof of Theorem \ref{t:W-Holder-MA}]
{\bf (h4) $\Longleftrightarrow$ (h1)}. Recall from \cite{KN} that (h1) is equivalent to the following property

\begin{itemize}
\item[(h4)'] {\it $\mu$ belongs to $\cH(\alpha)$ for some $\alpha>0$ and $\mu$ is $\PSH(X,\omega,[-1,0])$-H\"older continuous.}
\end{itemize} 
By Proposition \ref{p:h-alpha}, the property (h4) is equivalent to (h4)' and hence to (h1).

\smallskip\noindent
{\bf (h5) $\Longrightarrow$ (h4)}. 
Assume that the property (h5) holds, i.e., $\mu$ is $W^*(X)$-H\"older continuous. 
Consider smooth functions $f_1, f_2\in \PSH(X,\omega)$ with values in $[-1,1]$. By Lemma \ref{l:w-w-norm-x}, the norms $\|f_i\|_*$ are bounded by 2. 
By Definition \ref{d:W-Holder} applied to $(f_1-f_2)/4$, we have $|\mu(f_1-f_2)|\leq c' \|f_1-f_2\|_{L^1}^\alpha$ for some positive constants $c'$ and $\alpha$ independent of $f_i$. 

Consider now arbitrary functions $f_1,f_2\in \PSH(X,\omega,[-1,0])$. By Demailly's regularization theorem \cite{Demailly}, there are sequences of smooth functions $f_{i,k}$ in $\PSH(X,\omega)$ decreasing to $f_i$. For $k$ large enough, we have that $f_{i,k}$ has values in $[-1,1]$. So we can apply the above estimate for $f_{i,k}$ instead of $f_i$ and let $k$ tend to infinity. Thus, the estimate $|\mu(f_1)-\mu(f_2)|\leq c' \|f_1-f_2\|_{L^1}^\alpha$ holds for all $f_1,f_2\in \PSH(X,\omega,[-1,0])$. We conclude that 
(h5) $\Longrightarrow$ (h4) and this ends the proof of the theorem.

\smallskip\noindent
{\bf (h5) $\Longrightarrow$ (h1)}. It is possible to end the proof using a more direct argument where the reader will see the role of (h5) more clearly. Assume that the property (h5) holds. 
We want to prove the property (h1) which says that $\mu=(\ddc u+\omega)^n$ for some H\"older continuous $\omega$-psh function $u$. By Theorem \ref{t:MA-C0}, there is a continuous $\omega$-psh function $u$ such that 
$$(dd^cu+\omega)^n = \mu \quad \text{and} \quad  \sup_X u =0.$$
Note that in order to get the last identity it is enough to subtract from $u$ a constant. 

Following \cite{DDGKPZ14} consider for $0<\delta<1$ the regularization $\rho_{\delta }u$ of the $\omega$-psh function $u$ defined  by

\begin{equation*}
\rho_\delta u(z):=\delta ^{-2n}\int_{\zeta\in T_{z}X}
u({\exp}_z(\zeta))\rho\big(\delta^{-2}|\zeta|^2_{\omega }\big)\,dV_{\omega}(\zeta),
\end{equation*}
where $\zeta \mapsto {\exp}_z(\zeta)$ is the exponential map from the tangent space $T_zX$ of $X$ at $z$ to $X$ with respect to the Riemannian metric associated with the K\"ahler form $\omega$.
Here,  $|\zeta|_\omega$ denotes the norm of the vector $\zeta\in T_zX$, $dV_\omega(\zeta)$ the volume element induced by the measure $(dd^c |\zeta |_\omega^2)^n$ on the complex space $T_zX\simeq \C^n$, and the smoothing kernel  $\rho: \R_+\rightarrow\R_+$  is given by
$$\rho(t)=\begin{cases}\frac {\eta}{(1-t)^2}\exp(\frac 1{t-1})&\ {\rm if}\ 0\leq t\leq 1,\\0&\
{\rm if}\ t>1\end{cases}$$
 with a suitable constant $\eta$, such that
\begin{equation*}
\int_{\mathbb C^n}\rho(\|z\|^2)\,dV(z)=1,
\end{equation*}
where $dV$ denotes the Lebesgue measure in $\mathbb C^n$.

A basic result of Demailly \cite{Demailly} shows that 
$$dd^c \rho_\delta u (z)  \geq -  A \omega$$
for some constant $A >1$ depending only on the curvature of $(X,\omega)$.  Since $u$ is bounded, $\|\rho_\delta u\|_\infty$ is bounded by a constant independent of $\delta$. Hence, by Lemma \ref{l:w-w-norm-x}, 
since $A^{-1}\rho_\delta u \in \PSH(X,\omega)$, 
there is a positive constant $c_1$ independent of $\delta$ such that  $\|\rho_\delta u\|_* \leq c_1$.

By \cite{DDGKPZ14}, we already have $\|\rho_\delta u - u \|_{L^1} \leq c_2 \delta^2$ for some constant $c_2>0$. 
The $W^*(X)$-H\"older continuity of $\mu$ applied for $A^{-1}\rho_\delta u$ and $A^{-1}u$ implies 
$$\int_X |\rho_\delta u - u|  d\mu \leq c_3 A^{1-\alpha}  \|\rho_\delta u -u\|_{L^1}^\alpha \leq c_4 \delta^{2\alpha}$$
for some positive constants $\alpha$, $c_3$ and $c_4$. 
Note that the first inequality is a direct consequence of the $W^*(X)$-H\"older continuity of $\mu$ when $u$ is smooth. The same inequality holds for $u$ continuous because we can approximate it uniformly by smooth functions in $\PSH(X,\omega)$, see \cite{Demailly}. 
It follows that $u$ is H\"older continuous with the H\"older exponent $2\alpha/(n+1)$ (see e.g. \cite[Page 83]{DN} and \cite[Proposition~3.3]{DDGKPZ14}). The proof is complete.
\endproof

\begin{remark} \label{r:w-star} \rm
In the proof of (h5) $\Longrightarrow$ (h4), we see that the properties (h1)-(h6) are actually equivalent to the following property 
\begin{itemize}
\item[(h5)'] $\mu$ satisfies the estimate in Definition \ref{d:W-Holder} for $f$ in the set 
$$\big\{f\in W^*(X)\cap\Cc^\infty(X): \ \|f\|_*\leq 1, \quad f \text{ has values in } [0,1] \big\}.$$
\end{itemize}
\end{remark}

Let us give some applications of the new criteria (h5) and (h5)'. Denote by ${\rm MAH}(X,\omega)$ the set of all Monge-Amp\`ere measures with H\"older continuous potentials. It is proved in  \cite[Theorem B]{DDGKPZ14} that this is a convex set, and more generally, it satisfies the $L^p$-property for $p>1$. This result is also an immediate consequence of (h5)'.

\begin{remark}\label{r:holder-exponent} \rm 
By the proof of (h5) $\Longrightarrow$ (h1), if $\mu$ is $W^*$-H\"older continuous with exponent $\alpha>0$, then its Monge-Amp\`ere potential has the H\"older exponent $2\alpha/(n+1)$. This is probably not an optimal result.
\end{remark}

\begin{corollary}
Let $\mu \in {\rm MAH}(X,\omega)$ and $g \in L^p(\mu)$ a non-negative function with $p>1$ 
and $\int_X gd\mu =1$. Then, we have $g \mu \in {\rm MAH}(X, \omega)$.  
\end{corollary}

\begin{proof}
By Remark~\ref{r:w-star}, it is sufficient to show for some positive constants $c$ and $\alpha$ that
$\int_X f g d\mu \leq c \|f\|_{L^1}^\alpha $ for every $f\in W^*(X)\cap\Cc^\infty(X)$ with  $\|f\|_*\leq 1$ and $0\leq f \leq 1$.
Indeed, by the equivalence between (h1) and (h5), there are positive constants $c'$ and $\alpha'$ such that $\mu(f) \leq c' \|f\|_{L^1}^{\alpha'}$.
Now let $q\geq 1$ be the conjugate of $p$, i.e., $1/p+1/q=1$. The H\"older inequality gives 
$$\mu(fg) \leq \|f \|_{L^q(\mu)} \|g\|_{L^p(\mu)} \leq \|g\|_{L^p(\mu)} \|f\|_{L^\infty}^{1-1/q}  \mu(f)^{1/q} 
\leq c  \|f\|_{L^1}^{\alpha},$$
where $\alpha:=\alpha'/q$ and $c$ is a positive constant. This completes the proof.
\end{proof}

The following proposition is the product property for $W^*$-H\"older continuity. Together with our main results, it implies similar product properties for all (h1)-(h6).

\begin{proposition}
Let $X_i$ be a compact K\"ahler manifold of dimension $n_i$ and $\omega_i$ a K\"ahler form on $X_i$ normalized so that $\int_{X_i} \omega_i^{n_i} =1$, for $i=1,2$. Let $\mu_i$ be a probability measure on $X_i$. Then, the following properties are equivalent:

\begin{itemize}
\item[(i)] $\mu_i$ is $W^*(X_i)$-H\"older continuous for $i=1,2$;
\item[(ii)] $\mu := \mu_1 \times \mu_2$ is $W^*(X_1 \times X_2)$-H\"older continuous.
\end{itemize}
\end{proposition}

\begin{proof} 
Denote by $\pi_i:X_1\times X_2\to X_i$ the natural projection for $i=1,2$. Consider the K\"ahler form $\omega:=\lambda(\pi_1^*(\omega_1)+\pi_2^*(\omega_2))$ on $X$ with the normalization factor  $\lambda>0$ chosen so that $\int_X\omega^{n_1+n_2}=1$. The $W^*(X)$-H\"older continuity doesn't depend on this choice.

\smallskip\noindent
{\bf (i) $\Longrightarrow$ (ii).} Assume the property (i). We deduce from the equivalence (h5) $\Longleftrightarrow$ (h1) that $\mu_i = (dd^c u_i + \omega_i)^{n_i}$, where $u_i$ is a
H\"older continuous $\omega_i$-psh function on $X_i$ for $i=1,2$. It follows that 
$$\mu=\mu_1\times\mu_2= \big(\ddc (u_1\circ \pi_1)+\pi_1^*(\omega_1)\big)^{n_1}\wedge  \big(\ddc (u_2\circ \pi_2)+\pi_2^*(\omega_2)\big)^{n_2}.$$
Thus, (ii) follows from Proposition \ref{p:h2-h6}.

\smallskip\noindent
{\bf (ii) $ \Longrightarrow$ (i).} Assume (ii) is true. We only prove the property (i) for $\mu_1$ as the proof for $\mu_2$ is similar. Let $f_1 \in W^*(X_1) \cap \Cc^\infty(X_1)$ such that $\|f_1\|_*\leq 1$. 
Define $f:=f_1\circ \pi_1$. It is clear that $\|f\|_*$ is bounded. By the $W^*(X)$-H\"older property of $\mu$ applying for $f$ we have 
$$|\mu_1(f_1)| = |\mu(f)|    \leq  c \|f\|_{L^1(\mu)}^\alpha = c\|f_1\|_{L^1(\mu_1)}^\alpha$$for some positive constants $c$ and $\alpha$. Thus, $\mu_1$ is $W^*(X_1)$-H\"older continuous.
\end{proof}

\begin{remark} \rm \label{r:final}
Let $\mu$ be the measure as in Theorems \ref{t:main}, \ref{t:main-local}, \ref{t:DN-KN} or \ref{t:W-Holder-MA}.
According to these theorems, Lemma \ref{l:W-MA-C0} and Remark \ref{r:MA-int}, we can extend $\mu$ to a linear continuous functional on $W^*(X)$ or $W^*(\Omega)$ and the estimates in Definitions \ref{d:W-Holder} and \ref{d:W-Holder-local} are satisfied for every $f\in W^*(X)$ or $f\in W^*(\Omega)$ respectively. 
\end{remark}

\section{Comparison of two capacities}

Here we wish to show that a natural capacity for the  complex Sobolev space $W^*$,   introduced by Vigny \cite{Vig} and called "functional capacity", is well comparable to the Bedford-Taylor capacity  $\capa_\omega $.  

The functional capacity on a compact K\"ahler manifold $X$  is defined as follows. 
For a Borel set $E \subset X$ put
$$
	\mathcal{L} (E) := \big\{v \in W^* (X): v \leq -1 \quad \text{a.e. on some neighborhood of } E, v \leq 0 \text{ on }X \big\}.
$$
The {\it functional capacity} of the set $E$ is given by
$$
	C(E) := \inf \big\{ \|v\|_*^2 : v \in \mathcal{L} (E) \big\}.
$$
As noted in \cite{Vig}, taking $v=-1$ gives $C(E)\leq 1$. 
Recall from Section 3 that the Bedford-Taylor capacity of $E$, with respect to a K\"ahler form $\omega$,  is defined by
$$
	\capa_\omega (E) := \sup \Big\{\int_E (\omega + dd^c u)^n : u\in \PSH(\omega ), -1\leq u \leq 0 \Big\}.
$$

\begin{proposition} There exists a constant $A>0$ depending only on $(X, \omega )$ such that 
$$
	\frac{1}{A} \capa_\omega (E) \leq C(E) \leq A \left[ \capa_\omega (E) \right]^\frac{1}{n}.
$$
\end{proposition}

\begin{proof}
The first inequality was obtained by Vigny  \cite[Theorem~28]{Vig}, so we only deal with the second inequality.
Since both capacities are regular (see \cite[Theorem~4.2]{GZ05} and \cite[Theorem~30]{Vig})  we only need to show this inequality for any compact regular set $K \subset X$ instead of $E$. 
Here, the regularity of $K$ means that the extremal functions $h_K$ and $V_K$, defined in  Section 3, are continuous. Therefore,
$$
	\capa_\omega (K) = \int_K (\omega + dd^c h_K)^n = \int_X - h_K (\omega + dd^c h_K)^n.
$$
Set  $M:= \sup_X V_K\geq 0$. Without loss of generality, we may assume that $M<+\infty$, otherwise both capacities vanish (see \cite{ko05}, \cite{Vig}).
If $M \leq 1$, then $h_K = V_K -1$. By Stokes' theorem, the definition of $\capa_\omega(K)$ and the above identities, we have
$$
	\int_X \omega ^n = \int_K (\omega+ dd^c V_K)^n \leq \capa_\omega (K) \leq \int_X(\omega+\ddc h_K)^n=\int_X \omega ^n.
$$
Hence,
$\capa_\omega (K) = \capa_\omega (X)$ and  the desired inequality follows immediately.

It remains to consider the case $M \geq 1$. We have
$$
 	-1 \leq u:= \frac{V_K- M}{M} \leq h_K \leq 0.
$$
Hence, for any $\varepsilon >0$, one has $u -\varepsilon \in \mathcal{L} (K)$, see Lemma \ref{l:w-w-norm-x}. 
So, by the definition of $C(K)$, we have  $C(K)\leq \|u\|_*^2$.
We now estimate the  $W^*(X)$-norm of $u$.  Since $\sup_X (V_K - M) =0$,
$$
\|u\|_{L^1} =	\int_X -u \omega^n \leq \frac{A'}{M}  
$$
for some constant $A'>0$ independent of $u$ as $\PSH_0(X,\omega)$ is compact in $L^1(X)$.
Furthermore, since $V_K \in \PSH(\omega )$, we have $dd^c u \geq -\frac{\omega }{M} $ and thus
$$
	T := \frac{\omega}{M} + dd^c u +  \frac{1}{2} dd^c u^2  ={\omega\over M} +\Big(1-{u\over 2}\Big)\ddc u + du \wedge d^c u \geq du \wedge d^c u. 
$$
From those two facts and the definition of $\|u\|_*$, we get
$$
	\|u\|_*^2 \leq  2(\|u\|_{L^1}^2 + \|T\|) \leq  \frac{2A'^2}{M^2} + \frac{2}{M} \leq \frac{A}{M}
$$
for some constant $A>0$. 
The  global Alexander-Taylor inequality  (see Lemma \ref{l:extremal} (c) above)
says that 
$$
\frac{1}{M} \leq [\capa_\omega (K)]^\frac{1}{n},
$$
which, applied to the estimate for  $\|u\|_*^2 $, gives
$$
C (K) \leq \|u\|_*^2 \leq A [\capa_\omega (K)]^\frac{1}{n} .
$$
This is the desired inequality.
\end{proof}

\end{document}